\documentclass[11pt]{article}
\usepackage{amsmath,amssymb,amsthm, euscript, bm}

\usepackage{mathrsfs}

\usepackage{amscd}
\usepackage[all]{xy}

\usepackage{mdframed,xcolor}

\newtheorem{theorem}{Theorem}[section]
\newtheorem{proposition}[theorem]{Proposition}

\newtheorem{lemma}[theorem]{Lemma}
\newtheorem{corollary}[theorem]{Corollary}
\theoremstyle{definition}
\newtheorem{definition}[theorem]{Definition}

\newcommand{\N}{\mathbb{N}}
\newcommand{\Z}{\mathbb{Z}}
\newcommand{\Q}{\mathbb{Q}}
\newcommand{\R}{\mathbb{R}}
\newcommand{\C}{\mathbb{C}}

\newcommand{\tb}{\widetilde{b}}

\newcommand{\bpsi}{\overline{\psi}}

\newcommand{\biota}{\overline{\iota}}
\newcommand{\tiota}{\widetilde{\iota}}
\newcommand{\Js}{$\mathcal{Z}$}

\newcommand{\bN}{\overline{N}}
\newcommand{\tN}{\widetilde{N}}

\newcommand{\Cs}{$\mathrm{C}^*$-al\-ge\-bra}
\newcommand{\Css}{$\mathrm{C}^*$-sub\-al\-ge\-bra}

\newcommand{\tpsi}{{\widetilde{\psi}}}

\newcommand{\tp}{{\widetilde{p}}}

\newcommand{\tr}{{\mathrm{tr}}}

\newcommand{\tvarphi}{{\widetilde{\varphi}}}

\DeclareMathOperator{\Image}{Im}
\DeclareMathOperator{\id}{id}
\DeclareMathOperator{\Ad}{Ad}

\DeclareMathOperator{\End}{End}

\DeclareMathOperator{\Lip}{Lip}
\DeclareMathOperator{\supp}{supp}

\DeclareMathOperator{\rank}{rank}
\DeclareMathOperator{\diag}{diag}

\begin{document}
\title{Endomorphisms of \Js-absorbing  \Cs s\\ without conditional expectations}
\author{Yasuhiko Sato \thanks{The author was supported by JSPS KAKENHI Grant Numbers 19K03516 and the Department of Mathematical Sciences,  Kyushu University.}}
\date{}
\maketitle
\vspace{-20pt}
{\centering \emph{Dedicated to the memory of Eberhard Kirchberg}\par}

\begin{abstract} 
We construct an endomorphism of the Jiang-Su algebra \Js{} which does not admit a conditional expectation. This answers a question in the testamentary homework by E.\,Kirchberg. As an application, it is shown that any unital separable nuclear \Js-absorbing \Cs{} is non-transportable in the Cuntz  algebra $\mathcal{O}_2$.
\end{abstract}

\section{Introduction}\label{Sec1}
Recent developments in the theory of classifications of amenable \Cs s were achieved by numerous researchers, and consequently a very abstract class of unital separable simple \Cs s has now been classified by K-theoretical invariants; see \cite{K1}, \cite{K2}, \cite{KP}, \cite{P}, \cite{GLN}, \cite{GLN2}, \cite{TWW}, \cite{EN}, \cite{EGLN}, and \cite{EGLN2}. The Kirchberg-Phillips classification theorem has been recognized as one of the most successful results in the classification program established by G.\,A.\,Elliott. For an overview of the history and context of Elliott's classification programme, we refer the reader to \cite{ET}, \cite{Win2}, and \cite{Whi}.

As an application of the classification theorem, we are now in position to follow a typical but  crucial strategy consists of the following three steps: Here we assume that it is required to show a certain  property, denoted by property $X$, for a given unital separable simple classifiable \Cs{} $A$. 
The first step in the strategy is to construct an amenable \Cs{} $B$ with the  property $X$, which may be expected naturally. Second, by modifying $B$ and keeping the property $X$, it allows one to reconstruct a unital separable simple amenable \Cs{} $B_A$ with the property $X$ whose K-theoretical invariant coincides with that of $A$. In other words, at the revel of classification invariants $B_A$ is isomorphic to $A$. Finally, refining $B_A$ such that it is classifiable further, the classification theorem enables us to show that $A$ is isomorphic to $B_A$ as \Cs s, this leads to the desired property $X$ 
of $A$. If the property $X$ is compatible with tensor products, the tensor product with the Jiang-Su algebra \Js{} easily induces the classifiability, see \cite{MS}, \cite{SWW}, \cite{CETWW}.

This classification strategy can be applied to ``the existence of non-approximately inner flow'' or more generally ``the existence of a flow which realizes a possible KMS-bundle'' as the property $X$ in UHF algebras, which in fact provides a counter example to the Powers-Sakai conjecture \cite{Kis}, \cite{MS}, and a concrete realization of possible KMS-bundles on classifiable \Cs s  \cite{Thoms2}, \cite{ET}, \cite{EST}, \cite{ES}. 

In this paper, we apply the classification strategy to show the existence of a unital endomorphism without conditional expectations for classifiable \Cs s as follows, that has been studied in Kirchberg's manuscript \cite{EK}. 

\begin{theorem}\label{Thm1.1}
\begin{item}
\vspace{-15pt}
\item[\rm (i)] There exists an endomorphism of the Jiang-Su algebra \Js{} which does not have a conditional expectation.
\item[\rm (ii)] Suppose that $A$ and $B$ are unital nuclear \Js-absorbing $\mathrm{C}^*$-algebras and that there exists a unital embedding of $B$ into $A$. Then there exists a unital embedding of $B$ into $A$ which does not have a conditional expectation.
\end{item}
\end{theorem}

This theorem answers a question in \cite[p.795 and p.1220]{EK}, which asks whether the CAR algebra $M_{2^{\infty}}$ has a unital endomorphism with no conditional expectations (see Corollary \ref{Cor3.3}). In consequence of this, we see that the  Cuntz algebra $\mathcal{O}_2$ is not transportable in itself (see Theorem \ref{Thm4.2}). 

\medskip

Yuhei Suzuki has informed the author that his complete description of intermediate operator algebras given in \cite{Suz} provided an endomorphism of the Cuntz algebra $\mathcal{O}_{\infty}$ which does not admit a conditional expectation. 

\bigskip

Before going to the next section, let us collect some basic notations and terminologies. Throughout the paper, let $A^+$ denote the cone of positive elements in a \Cs{} $A$. In the case of a unital \Cs{} $A$, we denote by $1_A$ the unit of $A$. For $n\in\N$, let  $M_n$ be the \Cs{} of complex $n\times n$ matrices, $\{e_{i, j}^{(n)}\ :\ i, j=1,2,...,n\}$ the set of canonical matrix units of $M_n$, and $\tr_n$ the normalized trace on $M_n$. To simplify, we set $1_n =1_{M_n}$. 

For a \Cs{} $A$, we denote by $\id_A$ the identity map on $A$, i.e., $\id_A(a)=a$ for any $a\in A$. For two \Cs s $A$ and $B$, by an {\it embedding} $\varphi$ of $A$ into $B$ we mean an injective $*$-homomorphism from $A$ to $B$. 

A \emph{conditional expectation} $E$ for an embedding $\varphi$ of $A$ into $B$ is a completely positive contractive map from $B$ to $A$ such that $E\circ\varphi =\id_A$.  Because of Tomiyama's theorem (see \cite{Lan}, \cite{BO}), it is well-known that $E(\varphi(a)\, b\, \varphi(c))=a\,E(b)\,c$ for any $a$, $c\in A$, and $b\in B$. If an embedding $\varphi$ is unital (i.e., $\varphi(1_A)=1_B$), then so is a conditional expectation for $\varphi$.

\section{Jiang and Su's connecting maps reconstructed}\label{Sec2}
 In the original construction of the Jiang-Su algebra \Js{}  \cite[Proposition 2.5]{JS}, the connecting maps can not have conditional expectations, even in the constructions \cite[Theorem 4.1]{JS} and \cite[Theorem 2.1]{Ror} the existence of conditional expectations is not necessarily guaranteed (see Proposition \ref{Prop2.2} \rm{(i)}). For this reason, in Proposition \ref{Prop2.2} \rm{(ii)}, we shall modify the connecting maps with conditional expectations for the construction of \Js.
 
 First, we prepare a rather accurate notion for embeddings of prime dimension drop algebras which is an approximate formula of standard $*$-homomorphisms  in \cite[(2.1)]{Ror}.
 \begin{definition}\label{DefStand}
For relatively prime natural numbers $p_0$, $p_1$, and  for unital embeddings $\iota_i$ of $M_{p_i}$ into $M_{p_0}\otimes M_{p_1}$, $i=0, 1$, we denote by $I(p_0, p_1)$ the prime dimension drop algebra defined by 
\[ I(p_0, p_1)=\{ f\in C([0, 1])\otimes M_{p_0}\otimes M_{p_1}\ : \ f(i)\in \Image(\iota_i) \text{ for both } i=0, 1\}.\]
 From the definition, it follows that $I(p_0, p_1)$ is independent of the choice of unital embeddings $\iota_i$, $i=0, 1$. Let $\mu$ be the Lebesgue measure on $[0, 1]$. A unital $*$-homomorphism $\varphi$ from $I(p_0, p_1)$ into \Js{} is called \emph{standard} if 
 \[\tau_{\mathcal{Z}}\circ\varphi(f)=\int_{[0, 1]}\tr_{p_0p_1}(f(t))\ d\mu(t),\] 
 for any $f\in I(p_0, p_1)$, where $\tau_{\mathcal{Z}}$ is the unique tracial state of \Js. Let $\tp_0$ and $\tp_1$ be relatively prime natural numbers and $\varphi$  a unital embedding of $I(p_0, p_1)$ into $I(\tp_0, \tp_1)$. For a finite subset $F$ of $I(p_0, p_1)$ and $\varepsilon>0$, we say that $\varphi$ is \emph{$(F, \varepsilon)$-standard} if 
 \[ \rvert \tau\circ \varphi(f) -\int_{[0, 1]}\tr_{p_0p_1}(f(t))\ d\mu(t)\rvert< \varepsilon,\]
 for any tracial state $\tau$ of $I(\tp_0, \tp_1)$ and $f\in I(p_0, p_1)$.
  \end{definition}
 
 \begin{proposition}\label{Prop2.2} 
 Let $p_0$ and $p_1$ be relatively prime natural numbers.
 \begin{item}
\vspace{-15pt}
\item[\rm (i)] The unital embedding of $I(p_0, p_1)$ constructed in the same way as {\rm \cite[Proposition 2.5]{JS}} does not have a conditional expectation.
\item[\rm (ii)] For a finite subset $F$ of $I(p_0, p_1)$ and $\varepsilon>0$ , there exists a natural number $N$ satisfying that: for relatively prime natural numbers $\tp_0$ and $\tp_1$ with $\tp_i>N$ for $i=0, 1$, there exist a unital  $(F, \varepsilon)$-standard embedding $\varphi$ of $I(p_0, p_1)$ into $I(\tp_0, \tp_1)$ and  a conditional expectation for $\varphi$. 
\end{item}
 \end{proposition}
 \begin{proof}
 To simplify some notations we regard $i=0, 1$ as elements in $\Z/2\Z$, and set $d=p_0p_1$.
 
 \noindent {\rm (i)} Let $l_i$, $i=0, 1$ be natural numbers such that $l_i> 2p_{i+1}$ for both $i=0, 1$, and $l_0p_0$ and $l_1p_1$ are also relatively prime. Set $\tp_i=l_ip_i$ for $i=0, 1$, and $k=l_0l_1$. Recall that the connecting map $\varphi$ from $I(p_0, p_1)$ to $I(\tp_0, \tp_1)$ in the proof of \cite[Proposition 2.5]{JS} is constructed from continuous functions $\xi_j\in C([0, 1])$, $j=1,2,..., k$ (called a section of eigenvalues) defined by 
 \[\xi_j(t)= \begin{cases}
t/2,\quad \quad &1\leq j\leq r_0,  \\ 
1/2, \quad  \quad &r_0<j \leq k-r_1,\\
(t+1)/2, \quad \quad &k-r_1< j\leq k,
\end{cases}\]
where $r_0$ and $r_1$ are natural numbers such that $r_i<\tp_{i+1}$ and $\tp_{i+1}|k-r_i$ for $i=0$, $1$. Set $\xi(f)=\diag(f\circ\xi_1, f\circ\xi_2, ..., f\circ\xi_k)\in C([0, 1])\otimes M_{d}\otimes M_k$ for $f\in I(p_0, p_1)$. By the choice of $r_0$ and $r_1$, we obtain a unitary $u\in C([0, 1])\otimes M_d\otimes M_k$ such that $\Ad u\circ\xi(f)\in I(\tp_0, \tp_1)$ for any $f\in I(p_0, p_1)$. The unital embedding $\varphi$ was defined by $\varphi=\Ad u\circ\xi$. Since the definition of $I(\tp_0, \tp_1)$ is independent of the unital embeddings of $M_{\tp_i}$, $i=0, 1$ into $M_d\otimes M_k$, we may regard $\varphi=\xi$.

Assume that there exists a conditional expectation $E_{\varphi}$ for $\varphi$.
Let $g_i$, $i=0, 1$ be positive continuous functions on $[0, 1]$ such that $g_i(i)=1$ for both $i=0, 1$, $g_0|_{[1/2, 1]}=0$, and $g_1|_{[0, 1/2]}=0$. Define a positive element $F$ in $I(\tp_0, \tp_1)$ by 
\[ F(t) =(1-t)\varphi(g_0\otimes 1_d)(0) +t (1_{dk}-\varphi(g_1\otimes 1_d)(1)),\]
for $t\in [0, 1]$. For any $h\in C([0, 1])$ with $h|_{[1/2, 1]}=0$, we have 
\[E_{\varphi}(F)(h\otimes 1_d)=E_{\varphi}(F\,\varphi(h\otimes 1_d))=E_{\varphi}(\varphi(h\otimes 1_d))=h\otimes 1_d.\]
On the other hand, for any $h\in C([0, 1])$ with $h|_{[0, 1/2]}=0$, we have 
\[E_{\varphi}(F)(h\otimes 1_d)=E_{\varphi}(F\, \varphi(h\otimes 1_d))=E_{\varphi}(0)=0,\]
which contradicts the continuity of $E_{\varphi}(F)$ at $1/2$.

\noindent {\rm (ii)}  This argument is a variant of the proofs of \cite[Theorem 4.1]{JS} and \cite[Theorem 2.1]{Ror}. In order to obtain a conditional expectation, we shall insert in addition the following section $\xi_e$, and arrange the amplifications suitably.

Replacing $\varepsilon$ with $\varepsilon/\max\{\, \| f\|\, :\, f\in F\}$, we may assume that $F$ is a set of contractions in $I(p_0, p_1)$.
For a Lipschitz continuous function $f\in I(p_0, p_1)$, we denote by $\Lip(f)$ the Lipschitz constant of $f$. Since the set of all Lipschitz continuous functions is dense in $I(p_0, p_1)$, without loss of generality we may assume that $\Lip (f)<\infty$ for all $f\in F$. Let $\tN\in\N$ be such that $\displaystyle \max\{\Lip(f), 1\}< \tN \varepsilon/ 8$, and 
\[\left\lvert\int_{[0, 1]} \tr_d(f(t))\ d\mu(t) -\frac{1}{k}\sum_{j=1}^k \tr_d(f(j/k))\right\rvert< \varepsilon/2,\]
for all $f\in F$ and $k\geq \tN$.  Set $N=(\tN d)^2\in \N$.

Let $\iota_i$, $i=0, 1$ be the unital embeddings of $M_{p_i}$ into $M_d\cong M_{p_0}\otimes M_{p_1}$ which determine the prime dimension drop algebra $I(p_0, p_1)$. We suppose that $\tp_0$ and $\tp_1$ are relatively prime natural numbers such that $\tp_0$, $\tp_1>N$, and let $l_0$ and $l_1\in\N$ be such that $l_ip_i=\tp_i$ for both $i=0, 1$. Since the same argument works for the case $\tp_1>\tp_0$, in what follows we assume that $\tp_0> \tp_1$. Setting $k=l_0l_1$, we will prepare some natural numbers $y_j$ and $y_{j+1/2}$ inducing a partition of $k$ below. By the choice of $N$, there exist natural numbers $\bN_i \geq \tN$, $i=0, 1$ and $r_i$, $i=0, 1$ such that $l_i=(\bN_i-1)p_{i+1}+r_i$, and $0\leq r_i < p_{i+1}$ for both $i=0, 1$. Since $\tp_0=l_0p_0$ and $\tp_1=l_1p_1$ are relatively prime, it follows that $r_i>0$ for $i=0, 1$. We set $\bN=\bN_1$, and note that 
\[ k=l_0l_1=(\bN-1)\tp_0+l_0r_1=(\bN_0-1)\tp_1 +l_1r_0.\]
By $\tp_0>\tp_1$,
we have natural numbers $y_0$, $y_1$,..., $y_{\bN-1}$, $y_{1/2}$, $y_{3/2}$,..., $y_{\bN-3/2}$, and $c_1$, $c_2$,..., $c_{\bN-1}$ such that $y_j+y_{j+1/2}=\tp_0$ for $j=0, 1,..., \bN-2$, $y_{\bN-1} =l_0r_1$, $y_0=l_1r_0$, and $y_{j-1/2} +y_j=c_j\tp_1$ for $j=1,2, ..., \bN-1$. 
Note that 
\[\sum_{j=0}^{\bN-1} y_j+\sum_{j=1}^{\bN-1}y_{j-1/2} =(\bN-1)\tp_0+l_0 r_1=\left(\sum_{j=1}^{\bN-1} c_j\right)\tp_1 +l_1r_0=k.\]

The required unital embedding $\varphi$ of $I(p_0, p_1)$ into $I(\tp_0, \tp_1)$ can be constructed as the block diagonal map $\diag(f\circ\omega_1, f\circ\omega_2,..., f\circ\omega_k)$ for $f\in I(p_0, p_1)$ by setting the following continuous functions $\omega_j\in C([0, 1])$, $j=1,2,...,k$. For $j=0,1,..., \bN-1$, we define continuous functions $\xi_j$ and $\xi_{j+1/2}\in C([0, 1])$ by 
\[ \xi_j(t) = (t+j)/\bN,\quad \xi_{j+1/2} (t) = (j+1)/\bN\quad \text{for }t\in [0, 1],\]
 and $\xi_e=\id_{[0, 1]}\in C([0, 1])$. Set $z_{-1/2}=0$, and inductively $z_j= z_{j-1/2} + y_j$, and $z_{j+1/2}=z_j+y_{j+1/2}$ for $j=0, 1, ...\bN-1$.
We define $\omega_1=\xi_e$,  
\begin{align*} 
\omega_m&=\xi_j \quad \ \ \text{   for } j=0, 1,..., \bN-1, \text{ and } m\in \N \text{ with } z_{j-1/2}+2 \leq m \leq z_j,  \\
 \omega_m&=\xi_{j+1/2} \text{ for  }j=0, 1,..., \bN-2, \text{ and }  m\in \N \text{ with } z_j +1\leq m \leq z_{j+1/2} +1.
 \end{align*}
 As shown in the next paragraph, there exist  unital embeddings $\tiota_i$ of $M_{\tp_i}$ into $M_d\otimes M_k$, $i= 0, 1$ such that $\varphi(f)(i)\in \Image(\tiota_i)$ for any $f\in I(p_0, p_1)$ and $i=0, 1$. Thus we may regard $\varphi$ as a unital embedding of $I(p_0, p_1)$ into $I(\tp_0, \tp_1)$. 
 
 Let $g_j$, $j=0, 1, ..., \bN$ be positive continuous functions on $[0, 1]$ such that $g_j(j/\bN)=1$ and $\supp(g_j)\subset [(2j-1)/2\bN, (2j+1)/2\bN]\cap[0, 1]$.
 For $i=0, 1$ and $j=i, 1+i, 2+i, ..., \bN-1+i$, we let $h_j^{(i)}$ be the projection in $M_k$ such that $1_d\otimes h_j^{(i)}=\varphi(g_j\otimes 1_d)(i)$. Note that \[\rank(h_0^{(0)})=1+(y_0-1)=l_1r_0,\quad \rank(h_j^{(0)})= (y_{j-1/2}+1)+(y_j-1)=c_j\tp_1,\] for $j=1,2,...,\bN-1$.  Then there exist  unital embeddings $\iota_0^{(0)}$ of $M_{r_0}$ into $1_d\otimes h_0^{(0)}M_kh_0^{(0)}$ and $\iota_{0, j}$ of $M_{c_jp_1}$ into $(\Image(\iota_0)'\cap M_d)\otimes h_j^{(0)} M_k h_j^{(0)}$, $j=1,2,..., \bN-1$ such that 
 \[ \iota_0^{(0)}(e_{1, 1}^{(r_0)})\, 1_d\otimes e_{1, 1}^{(k)}= 1_d\otimes e_{1, 1}^{(k)},\quad \iota_{0, j}(M_{p_1}\otimes 1_{c_j})=(\Image(\iota_0)'\cap M_d)\otimes h_j^{(0)}.\]
 In the matrix algebra $(\Image(\iota_0)'\cap M_d)\otimes M_k\cong M_{p_1k}$, it follows that 
 \[\rank(\iota_0^{(0)}(e_{1,1}^{(r_0)}))=(l_1p_1r_0)/r_0=\tp_1,\quad \rank(\iota_{0, j}(e_{1, 1}^{(c_jp_1)}))=c_j\tp_1/c_j=\tp_1,\]
 for $j=1, 2, ..., \bN-1$. By $\displaystyle r_0+\left(\sum_{j=1}^{\bN-1}c_j\right)p_1=l_0$, we can find a unital embedding $\biota_0$ of $M_{l_0}$ into 
 $(\Image(\iota_0)'\cap M_d)\otimes M_k$ such that $\Image(\iota_0^{(0)})\cup \bigcup_{j=1}^{\bN-1}\Image(\iota_{0, j})\subset \Image(\biota_0)$ and 
 \[ \biota_0(e_{1, 1}^{(l_0)})\, (1_d\otimes e_{1, 1}^{(k)}) = \iota_0^{(0)}(e_{1,1}^{(r_0)})(1_d\otimes e_{1, 1}^{(k)}) = 1_d\otimes e_{1, 1}^{(k)}.\]
 In a similar way, we see that $\rank(h_j^{(1)})=(y_{j-1}-1)+(y_{j-1/2}+1)=\tp_0$ for $j=1, ..., \bN-1$, and $\rank(h_{\bN}^{(1)})=(y_{\bN-1}-1)+1=l_0r_1$. Recall that the position of $\omega_1=\xi_e$ in the definition of $\varphi$ is $1_d\otimes e_{1, 1}^{(k)}$, which implies that $ e_{1, 1}^{(k)}\, h_{\bN}^{(1)}=e_{1, 1}^{(k)}$. Then there exist unital embeddings $\iota_1^{(1)}$ of $M_{r_1}$ into $1_d\otimes h_{\bN}^{(1)}M_kh_{\bN}^{(1)}$ and $\iota_{1, j}$ of $M_{p_0}$ into $(\Image(\iota_1)'\cap M_d)\otimes h_j^{(1)}$, $j=1,2,...,\bN-1$ such that 
 \[\iota_1^{(1)}(e_{1, 1}^{(r_1)})(1_d\otimes e_{1, 1}^{(k)})=1_d\otimes e_{1, 1}^{(k)},\quad \rank(\iota_{1, j}(e_{1, 1}^{(p_0)}))=\rank(\iota_1^{(1)}(e_{1, 1}^{(r_1)}))=\tp_0\]
 in $(\Image(\iota_1)'\cap M_d)\otimes M_k$ for $j=1, 2,..., \bN-1$.
By $r_1+(\bN-1)p_0=l_1$, we can also find a unital embedding $\biota_1$ of $M_{l_1}$ into $(\Image(\iota_1)'\cap M_d)\otimes M_k$ such that $\Image(\iota_1^{(1)})\cup\bigcup_{j=1}^{\bN-1}\Image(\iota_{1, j})\subset \Image(\biota_1)$ and 
\[ \biota_1(e_{1, 1}^{(l_1)})(1_d\otimes e_{1, 1}^{(k)})=\iota_1^{(1)}(e_{1, 1}^{(r_1)})(1_d\otimes e_{1, 1}^{(k)}) =1_d\otimes e_{1, 1}^{(k)}.\]
For each $i=0, 1$, define a unital embedding $\tiota_i$ of $M_{\tp_i}$ into $M_{dk}\cong M_{\tp_0}\otimes M_{\tp_1}$ by 
\[\tiota_i(a\otimes b) =(\iota_i(a)\otimes 1_k)\, \biota_i(b)\]
for  $a\in M_{p_i}$ and $b\in M_{l_i}$. Set $p_j^{(0)}=c_jp_1$ for $j=1, 2, ...,\bN-1$ and $p_j^{(1)}=p_0$. Because of 
\[ (\Image(\iota_i)\otimes 1_k)((\Image(\iota_i)'\cap M_d)\otimes 1_k)\, \iota_{i, j}(1_{p_j^{(i)}})\subset\Image(\tiota_i),\]
for all $i=0, 1$ and $j=1, 2, ..., \bN-1$, thus we conclude that 
\[\varphi(f)(i)=(f(i)\otimes 1_k)\, \iota_i^{(i)}(1_{r_i})+\sum_{j=1}^{\bN-1}(f(j/\bN)\otimes 1_k)\, \iota_{i, j}(1_{p_j^{(i)}})\in \Image(\tiota_i),\]
for any $i=0, 1$ and $f\in I(p_0, p_1)$. 

Next, we show that $\varphi$ is $(F, \varepsilon)$-standard. For all $f\in F$, from $\Lip(f)<\varepsilon\bN/8$ it follows that 
\[\max_{t\in [0, 1]}\left\lVert f\circ \omega_m(t) - f(j/\bN)\right\rVert <\varepsilon /8,\]
for all $j=1,2,..., \bN-1$ and $m\in \N\setminus\{1\}$ with $(j-1)\tp_0+1\leq m\leq j\tp_0$. For any tracial state $\tau $ of $I(\tp_0, \tp_1)$, one obtain a probability measure $\mu_{\tau}$ on $[0, 1]$ such that
\[\tau(g)=\int_{[0, 1]}\tr_{dk}(g(t))\ d\mu_{\tau}(t), \]
for any $g\in I(\tp_0, \tp_1)$ (see \cite[Lemma 2.4]{JS} for example). Because of $\tp_i>(\tN d)^2$ and 
$(l_0r_1+\tp_0)/k< 2p_0/l_1 < \varepsilon/8$, it follows that  for any tracial state $\tau$ of $I(\tp_0, \tp_1)$ and $f\in F$, 
\begin{align*}
\tau\circ\varphi(f) &=\int_{[0, 1]} \tr_d\otimes \tr_k(\varphi(f)(t))\ d\mu_{\tau}(t) \\
&\approx_{\frac{\varepsilon}{8}}\frac{1}{k}\sum_{m=2}^{k-l_0r_1}\int_{[0, 1]}\tr_d(f\circ\omega_m(t))\ d\mu_{\tau}(t) \\
&\approx_{\frac{\varepsilon}{4}}\frac{\tp_0}{k}\sum_{j=1}^{\bN}\tr_d(f(j/\bN))\approx_{\frac{\varepsilon}{8}}\frac{1}{\bN}\sum_{j=1}^{\bN}\tr_d(f(j/\bN)) \\
&\approx_{\frac{\varepsilon}{2}}\int_{[0, 1]}\tr_d(f(t))\ d\mu(t).
\end{align*}

Finally, it remains to show that there exists a conditional expectation for $\varphi$. 
We define a unital completely positive map $E_{\varphi}$ from $I(\tp_0, \tp_1)$ into $C([0, 1])\otimes M_d$ by
\[E_{\varphi}=k(\id_{C([0, 1])\otimes M_d}\otimes\, \tr_k)\circ\Ad(1_{C([0, 1])\otimes M_d}\otimes e_{1, 1}^{(k)}).\]
It is straightforward to check that $E_{\varphi}\circ\varphi=\id_{I(p_0, p_1)}$. In the definition of  $\biota_i$, we have seen that $(1_d\otimes e_{1, 1}^{(k)})\biota_i(e_{1, 1}^{(l_i)})=1_d\otimes e_{1, 1}^{(k)}$ for both $i=0, 1$. Then for any $g\in I(\tp_0, \tp_1)$, it follows that $\Ad(1_d\otimes e_{1, 1}^{(k)})(g(i))\in \Image(\iota_i)\otimes e_{1, 1}^{(k)}$, which implies that 
$E_{\varphi}(g)(i)\in \Image(\iota_i)$ for both $i=0, 1$. Hence we may regard $E_{\varphi}$ as a conditional expectation for $\varphi$
 \end{proof}

 \section{Expectationless embeddings}\label{Sec3}
 The construction of the endomorphism in Theorem \ref{Thm1.1} will be derived from the basic observation of a continuous path of two embeddings below.
 
Let $A$, $B$, and $B_i$, $i=0, 1$ be unital \Cs s. For $i=0, 1$, we let $\iota_i$  be a unital embedding of $B_i$ into $B$ and $\psi_i$ a unital embedding of $A$ into $\Image(\iota_i)$. Assume that $\psi_0$, and $\psi_1$ are asymptotically unitarily equivalent by a continuous path $u_t$, $t\in[0, 1)$ of unitaries in $B$ with $u_0=1_B$ 
(that is $\lim_{t\to 1}\Ad u_t\circ\psi_0(a)=\psi_1(a)$ in the norm topology of $B$ for any $a\in A$).
We define a generalized dimension drop algebra $I(B_0, B_1)$ as the \Css{} of $C([0, 1])\otimes B$ consisting of all functions $f$ such that $f(i)\in \Image(\iota_i)$ for both $i=0, 1$. Thus we obtain a unital embedding $\psi$ of $A$ into $I(B_0, B_1)$ defined by 
  \[\psi(a)\, (t)= \begin{cases}
\Ad u_t\circ\psi_0(a),\quad\ \quad &t\in[0, 1),  \\ 
\psi_1(a),  \quad &t=1,\quad \text{ for } a\in A.
\end{cases}\]
\begin{lemma}\label{Lem3.1}
In the setting above, additionally suppose that $A$ is simple. For each $i=0, 1$, if $\psi_i$ does not have a conditional expectation, then neither does $\psi$.
\end{lemma}
\begin{proof}
Assume that there exists a conditional expectation $E$ for $\psi$. For $f\in C([0, 1])$, it follows that $E(f\otimes 1_B)\in A'\cap A=\C1_A$. Regarding $E$
 as a state of $C([0, 1])$, we obtain the probability measure $\mu_{E}$ on $[0, 1]$ which is determined by 
 \[\int_{[0, 1]}f(t)\ d\mu_E(t)\cdot 1_A= E(f\otimes 1_B)\quad\text{ for } f\in C([0, 1]).\]
 Denote by $z$ the positive continuous function on $[0, 1]$ defined by $z(t)=1-t$ for $t\in[0, 1]$.  For $b\in B_0$, we set $\tb(t)=z(t)\Ad u_t\circ\iota_0(b)\in B$, $t\in [0, 1)$ and $\tb(1)=0$. Then it follows that $\tb\in I(B_0, B_1)$. 
 
 In the case $\mu_E(z)=E(z\otimes 1_B)>0$, we can define a linear map $E_0$ from $B_0$ to $A$ by 
 \[E_0(b)=\mu_E(z)^{-1}E(\tb)\quad\text{ for } b\in B_0.\]
 It is straightforward to show that $E_0\circ\iota_0^{-1} : \Image(\iota_0)\rightarrow A$ is a conditional expectation for $\psi_0$, which is a contradiction. 
 
 Assume that $\mu_E(z)=0$. Since $E((1-z)^n\otimes \iota_1(b))\in A$, $n\in\N$ is a Cauchy sequence in the norm topology for any $b\in B_1$, we obtain a completely positive map $E_1$ from $B_1$ to $A$ defined by 
 \[E_1(b)=\lim_{n\to\infty}E((1-z)^n\otimes \iota_1(b))\quad\text{ for }b\in B_1.\]
 By $\mu_E(z)=0$, the map $E_1\circ\iota_1^{-1} : \Image(\iota_1)\rightarrow A$ is a conditional expectation for $\psi_1$, which contradicts the assumption. 
\end{proof}
 
 By applying the above lemma, in order to construct the required endomorphism of \Js{} in Theorem \ref{Thm1.1}, it suffices to find two unital expectationless embeddings of \Js{} into relatively prime UHF-algebras. 
 The following construction of unital embeddings can be regarded as a suitable modification of \cite[Proposition 2.3]{Sat} with conditional expectations.
 \begin{proposition}\label{Prop3.2}
 There exist two UHF-algebras $B_i$, $i=0, 1$ and unital embeddings $\iota_i$ of \Js{} into $B_i$, $i=0, 1$ such that the supernatural numbers of $B_0$ and $B_1$ are relatively prime and $\iota_i$ does not have a conditional expectation for both $i=0, 1$.
 \end{proposition}
 \begin{proof}
 To simplify notations, we regard $i=0, 1$ as elements in $\Z/2\Z$. 
 We let $p_n^{(i)}$, $q_n^{(i)}$, $n\in\N$, $i=0, 1$ be sequences of mutually relatively prime natural numbers such that
 \[p_n^{(i)}|p_{n+1}^{(i)},\quad q_n^{(i)}|q_{n+1}^{(i)},\quad p_n^{(i+1)}|(p_{n+1}^{(i)}/p_n^{(i)}-1), \quad q_n^{(i+1)}|(q_{n+1}^{(i)}/q_n^{(i)}-1),\]
 for any $n\in\N$ and $i=0, 1$. For example, such natural numbers are obtained by $p_1^{(0)}=2$, $p_1^{(1)}=3$, $q_1^{(0)}=5$, $q_1^{(1)}=7$, 
 \begin{align*}
 p_{n+1}^{(0)}&=p_n^{(0)}(L_np_n^{(1)}q_n^{(0)}q_n^{(1)}+1),\quad p_{n+1}^{(1)}=p_n^{(1)}(L_np_{n+1}^{(0)}q_n^{(0)}q_n^{(1)}+1),\\
  q_{n+1}^{(0)}&=q_n^{(0)}(L_np_{n+1}^{(0)}p_{n+1}^{(1)}q_n^{(1)}+1),\quad q_{n+1}^{(1)}=q_n^{(1)}(L_np_{n+1}^{(0)}p_{n+1}^{(1)}q_{n+1}^{(0)}+1),
 \end{align*}
 for some natural numbers $L_n$, $n\in\N$ inductively.
 Let $\varepsilon_n>0$, $n\in\N$ be a decreasing sequence which converges to $0$. Set $I_n=I(p_n^{(0)}, p_n^{(1)})$ and $J_n=I(q_n^{(0)}, q_n^{(1)})$, 
 $n\in\N$. Taking large natural numbers $L_n$, $n\in\N$ and applying Proposition \ref{Prop2.2} {\rm (ii)} inductively, for each $n\in\N$ we can obtain  increasing sequences $F_{n, m}$ and $G_{n, m}$, $m\in\N$ of finite subsets in $I_n$ and $J_n$, and  unital embeddings $\varphi_n :I_n\rightarrow I_{n+1}$ and $\psi_n : J_n\rightarrow J_{n+1}$ satisfying the following properties:
 \begin{itemize} 
 \item[\rm (i)]
 $\bigcup_{m=1}^{\infty} F_{n, m}$ and $\bigcup_{l=m}^{\infty} G_{n, m}$ are dense in the unit ball of $I_n$ and $J_n$, 
 \item[\rm(ii)]$\bigcup_{l, m=1}^{n} \varphi_{n, m}(F_{m, l})\subset F_{n+1, n+1}$, \quad $\bigcup_{l, m=1}^n \psi_{n, m}(G_{m, l})\subset G_{n+1, n+1}$, \\
 where $\varphi_{n, m}$ and $\psi_{n, m}$, $n>m$, denote the composed connecting maps defined by
\[ \varphi_{n, m}=\varphi_n\circ\varphi_{n-1}\circ\cdots\circ\varphi_m,\quad
\psi_{n, m}=\psi_n\circ\psi_{n-1}\circ\cdots\circ\psi_m,\]

\item[\rm (iii)]
$\varphi_n$ is $(F_{n, n}, \varepsilon_n)$-standard, $\psi_n$ is $(G_{n, n}, \varepsilon_n)$-standard, 
\item[\rm (iv)]
there exist conditional expectations for $\varphi_n$ and $\psi_n$.
\end{itemize}

Let $B_0$ and $B_1$ are UHF-algebras whose supernatural numbers are 
\[p_1^{(0)}p_1^{(1)}\prod_{n\in\N} \frac{p_{n+1}^{(0)}p_{n+1}^{(1)}}{p_n^{(0)}p_n^{(1)}}\quad\text{ and }\quad q_1^{(0)}q_1^{(1)}\prod_{n\in\N} \frac{q_{n+1}^{(0)}q_{n+1}^{(1)}}{q_n^{(0)}q_n^{(1)}}.\]
We shall construct the required embeddings $\iota_i$ just for $i=0$. By replacing $p_n^{(i)}$, $I_n$, $F_{n, m}$, and $\varphi_n$ with $q_n^{(i)}$, $J_n$, $G_{n, m}$, and $\psi_n$, the same argument shows the existence of $\iota_1$.

To simplify our notations, we set $l_n^{(i)}=p_{n+1}^{(i)}/p_n^{(i)}$, $i=0, 1$, $m_n=l_n^{(0)}l_n^{(1)}$, $d_n=p_n^{(0)}p_n^{(1)}$, and $C_n=C([0, 1])\otimes M_{d_n}$ for $n\in\N$. The construction of $\varphi_n$ in the proof of Proposition \ref{Prop2.2} \rm{(ii)} is determined by the section of eigenvalues $\omega_j\in C([0, 1])$, $j=1,2,...,m_n$. By using the same $\omega_j$, we define unital embeddings $\tvarphi_n $ of $C_n$ into $C_{n+1}$ by 
\[\tvarphi_n(f)=\diag(f\circ\omega_1, f\circ\omega_2, ..., f\circ\omega_{m_n})\quad\text{ for } f\in C_n.\]
It is clear that $\tvarphi_n|_{I_n}=\varphi_n$ for any $n\in\N$.

We denote by $\displaystyle\lim_{\longrightarrow}(I_n,\varphi_n)$ the inductive limit \Cs{} determined by $I_n$ and $\varphi_n$, $n\in\N$. For $n>m$, set $\varphi_{n,m}=\varphi_n\circ\varphi_{n-1}\circ\cdots\circ\varphi_m$ and let $\varphi_{\infty, n}$ be the induced map of $\{\varphi_n\}_{n\in\N}$ from $I_n$ to $\displaystyle\lim_{\longrightarrow}(I_n,\varphi_n)$. Since $\varphi_n$ is $(F_{n, n}, \varepsilon_n)$-standard for any $n\in\N$, it follows that the inductive limit \Cs{} $\displaystyle\lim_{\longrightarrow}(I_n,\varphi_n)$ has a unique tracial state. Actually, if $\tau$ and $\sigma$ are two tracial states of $\displaystyle\lim_{\longrightarrow}(I_n,\varphi_n)$, then it follows that 
\begin{align*}
\tau\circ\varphi_{\infty, n}(f)&=\tau\circ\varphi_{\infty, m+1}\circ\varphi_{m, n}(f)\approx_{\varepsilon_m}\int_{[0, 1]}\tr_{d_m}(\varphi_{m-1, n}(f)(t))\ d\mu(t)\\
&\approx_{\varepsilon_m}\sigma\circ\varphi_{\infty, m+1}\circ\varphi_{m, n}(f)=\sigma\circ\varphi_{\infty, n}(f),\quad \text{for }f\in F_{n, n}.
\end{align*}
Since $\varepsilon_m$, $m\in\N$ converges to $0$, we have $\tau\circ\varphi_{\infty, n}(f)=\sigma\circ\varphi_{\infty, n}(f)$ for all $f\in F_{n, n}$. From the definition of $F_{n, n}$, it is not so hard to see that $\bigcup_{n=1}^{\infty}\varphi_{\infty, n}(F_{n, n})$ is dense in the unit ball of $\displaystyle\lim_{\longrightarrow}(I_n, \varphi_n)$, which implies that $\tau=\sigma$. By the construction of $\varphi_n$, the simplicity of $\displaystyle\lim_{\longrightarrow}(I_n,\varphi_n)$ follows from a standard argument using \cite[Lemma III.4.1]{Dav}. Then the classification theory of \cite[Theorem 6.2]{JS} allows us to assert that $\displaystyle\lim_{\longrightarrow}(I_n,\varphi_n)$ is isomorphic to \Js.

For the same reason, we also see that the inductive limit \Cs{} $\displaystyle\lim_{\longrightarrow}(C_n, \tvarphi_n)$ has a unique tracial state. Thus, because of \cite[Theorem 1.4]{Thoms}, it follows that 
$\displaystyle\lim_{\longrightarrow}(C_n, \tvarphi_n)$ is a UHF-algebra.  
Since it is well-known that $K_0(C_n)\cong K_0(M_{d_n})$ for any $n\in\N$ as  ordered abelian groups, then the classification theorem of \cite{Ell} allows us to assert that $\displaystyle\lim_{\longrightarrow}(C_n, \tvarphi_n)$ is isomorphic to $B_0$ (A slightly different argument is found in \cite[Proposition 2.3]{Sat}).

For $n\in\N$, let $\eta_n$ be the canonical embedding of $I_n$ into $C_n$.
One can check that $\eta_n$ does not have a conditional expectation for any $n\in\N$. Now the following diagram commutes:

 \begin{equation*}
\xymatrix{
C_1\ar[r]^-{\tvarphi_{1}}&C_2\ar[r]^-{\tvarphi_{2}} &C_3\ar[r]&\quad\cdots \quad\ar[r]&C_n\ar[r]^-{\tvarphi_{n}}&C_{n+1}\ar[r]^-{\tvarphi_{n+1}}&\quad \cdots\\
I_1\ar[u]^-{\eta_{1}}\ar[r]_{\varphi_{1}}&I_2\ar[u]^-{\eta_{2}}\ar[r]_{\varphi_{2}} &I_3\ar[u]^-{\eta_{3}}\ar[r]&\quad\cdots\quad\ar[r] &I_n\ar[u]^-{\eta_{n}}\ar[r]_{\varphi_{n}}&I_{n+1}\ar[u]^-{\eta_{n+1}}\ar[r]_{\varphi_{n+1}}&\quad \cdots.
}
\end{equation*}

We let $\iota_0$ be the induced map of $\eta_n$, $n\in\N$ which is determined by $\iota_0(\varphi_{\infty, n}(a))=\tvarphi_{\infty, n}\circ\eta_n(a)$ for $a\in I_n$. This $\iota_0$ is a unital embedding of \Js{} into $B_0$ which does not have a conditional expectation. Because, considering  conditional expectations $E_n$ for $\varphi_n$, $n\in\N$, we obtain a conditional expectation $E$ for $\varphi_{\infty, 1}$ defined by 
\[ E(\varphi_{\infty, n}(a))=E_1\circ E_2\circ\cdots\circ E_{n-1}(a)\quad\text{ for }n\in\N\text{ and } a\in I_n.\]
If there exists a conditional expectation $E_{\iota_0}$ for $\iota_0$,
then the composition map $E\circ E_{\iota_0}\circ\tvarphi_{\infty, 1}$ can be a conditional expectation for $\eta_1$, which is a contradiction. 
\end{proof}

\begin{proof}[Proof of Theorem 1.1]
\

\noindent {\rm (i)}
By Proposition \ref{Prop3.2}, we obtain two UHF-algebras $B_i$, $i=0, 1$
whose supernatural numbers are relatively prime, and unital embeddings 
$\bpsi_i$, $i=0, 1$ of \Js{} into $B_i$  without conditional expectations. Although one can show that $\bpsi_0$ and $\bpsi_1$ are asymptotically unitarily equivalent using \cite[Lemma 5.3]{LN}, we take an elementary approach for the sake of self-contained presentations. Considering the infinite tensor product $\bigotimes_{n\in\N} B_i$ instead of $B_i$ for both $i= 0, 1$, we may assume that $B_i$, $i=0, 1$ are UHF-algebras of infinite type. Let $\Phi_i$, $i=0, 1$ be isomorphisms from $B_i\otimes B_i$ onto $B_i$. We identify $B_0\otimes B_1$ with $B_1\otimes B_0$ without mentioning, and let $B=B_0\otimes B_1$.  For $i=0, 1$, we let $\iota_i$ be the canonical embedding of $B_i$ into $B$. Set unital embeddings $\tpsi_i$, $i=0, 1$ of $\mathcal{Z}\otimes B$ into $B$ by 
\[ \tpsi_i(a\otimes b_0\otimes b_1)=\Phi_i(\bpsi_i(a)\otimes b_i)\otimes b_{i+1}\quad \text{for }a\in \mathcal{Z},\ b_i\in B_i,\]
where we regard $i=0, 1$ as elements in $\Z/2\Z$. Since $\tpsi_0$ and $\tpsi_1$ are unital endomorphisms of the UHF-algebra $B\cong \mathcal{Z}\otimes B$, then they are asymptotically unitarily equivalent, see \cite[Theorem 3.1]{Bl} or \cite[Proposition 1.3.4]{Ror0}. We define unital embeddings $\psi_i$, $i=0, 1$ of $\mathcal{Z}$ into $\Image(\iota_i)$ by $\psi_i(a)=\tpsi_i(a\otimes 1_B)$ for $a\in \mathcal{Z}$. Then $\psi_0$ and $\psi_1$ are also asymptotically unitarily equivalent in $B$. Since $\bpsi_i$ does not have a conditional expectation, neither does $\psi_i$ for both $i=0, 1$. Since any unitary in $B$ is homotopic
 to $1_B$, these $\iota_i$ and $\psi_i$, $i=0, 1$ satisfy the assumptions for Lemma \ref{Lem3.1}. Then there exists a unital embedding $\psi$ of \Js{} into $I(B_0, B_1)$ 
which does not have a conditional expectation. By \cite[Proposition 3.3]{RW}, there exists a unital embedding $\iota$ of $I(B_0, B_1)$ into \Js. Therefore, the composition map $\varphi=\iota\circ \psi$ is a unital endomorphism of \Js{} which does not have a conditional expectation. Indeed, if there exists a conditional expectation $E_{\varphi}$ for $\varphi$, then  $E_{\psi}=E_{\varphi}\circ\iota$ is so for $\psi$.
\medskip

\noindent {\rm (ii)}
Let $\iota$ be a unital embedding of $B$ into $A$, and let $\varphi$ the endomorphism of \Js{} obtained in {\rm (i)}. Then the tensor product $\eta=\iota\otimes\varphi$ is a unital embedding of $B\cong B\otimes \mathcal{Z}$ into $A\cong A\otimes \mathcal{Z}$. If there exists a conditional expectation $E_{\eta}$ for $\eta$, then taking a state $\omega$ of $B$ we can define a unital completely positive map $E_{\varphi}$ from \Js{} into \Js{} by $E_{\varphi}(a)=(\omega\otimes \id_{\mathcal{Z}})\circ E_{\eta}(1_A\otimes a)$ for $a\in$ \Js. It is straightforward to check that $E_{\varphi}$ is a conditional expectation for $\varphi$, which contradicts the assumption of $\varphi$ in {\rm (i)}. 
 \end{proof}
 
 \begin{corollary}\label{Cor3.3}
 \begin{item}
\vspace{-15pt}
\item[\rm (i)] For a unital nuclear \Js-absorbing \Cs{}, there exists a unital endomorphism of it which does not have a conditional expectation. In particular, UHF-algebras $M_{n^{\infty}}$, $n\in\N\setminus\{1\}$, the Cuntz algebras $\mathcal{O}_n$, $n\in\N\cup\{\infty\}\setminus\{ 1 \}$, and the irrational rotation algebras $\mathcal{A}_{\theta}$, $\theta\in\R\setminus\Q$ have such unital endomorphisms.

\item[\rm (ii)] For any supernatural number $\frak n$, there exists a unital embedding of \Js{} into the UHF-algebra of type $\frak n$ which does not have a conditional expectation. In other words, for any given pair of relatively prime supernatural numbers we have two UHF-algebras and unital embeddings satisfying the condition in Proposition \ref{Prop3.2}.

\item[\rm (iii)]
For relatively prime supernatural numbers $\frak p$ and $\frak q$, there exists a unital embedding of $I(M_{\frak p}, M_{\frak q})$ into \Js{} which does not have a conditional expectation. (This result is a variation of \cite[Proposition 3.3]{RW} involving the expectationless condition.)
\end{item}
 \end{corollary}
 \begin{proof}
 \noindent{\rm (i)}
 Considering a unital endomorphism $\varphi$ of \Js{} which does not have a conditional expectation by Theorem \ref{Thm1.1} {\rm (i)}, for a unital nuclear \Js-absorbing \Cs{} $A$ it is obvious that the endomorphism $\id_A\otimes \varphi$ of $A\otimes \mathcal{Z}\cong A$ does not have a conditional expectation. 
 
 \noindent{\rm (ii)} 
 It is well-known that both the UHF-algebra $M_{\frak n}$ of type ${\frak n}$ and \Js{} absorb  \Js{} tensorially \cite{JS}. By taking a unital embedding of \Js{} into $M_{\frak n}$ (for example the canonical embedding $\mathcal{Z}\rightarrow\mathcal{Z}\otimes M_{\frak n}\cong M_{\frak n}$), the statement follows from Theorem \ref{Thm1.1} {\rm (ii)} directly.

\noindent{\rm (iii)}
By \cite[Corollary 3.2]{RW} and \cite{JS}, we see that  both $I(M_{\frak p}, M_{\frak q})$ and \Js{} absorb \Js{} tensorially. Furthermore, from \cite[Proposition 3.3]{RW} there exists a unital embedding of $I(M_{\frak p}, M_{\frak q})$ into \Js. Then the statement also follows from Theorem \ref{Thm1.1} {\rm (ii)}. 
 \end{proof}
 
 For a \Cs{} $A$, we denote by $\End(A)$ the set of all endomorphisms of $A$ and equip $\End(A)$ with the point-norm topology in the following application of Theorem \ref{Thm1.1}.
 
 \begin{corollary}\label{Cor3.4}
 For a separable \Js{}-absorbing \Cs{} $A$, the set of endomorphisms of $A$ without conditional expectation is dense in $\End(A)$.
 \end{corollary}
 \begin{proof}
 Let $\Phi$ be an isomorphism from $A$ onto $A\otimes \mathcal{Z}$, and let $\varphi$ the unital endomorphism of \Js{} obtained in Theorem \ref{Thm1.1} {\rm (i)}. For $\alpha\in\End(A)$, a finite subset $F$ of $A$, and $\varepsilon >0$, it suffices to find $\widetilde{\alpha}\in \End(A)$ which does not have a conditional expectation but has $\left\|\widetilde{\alpha}(x)-\alpha(x)\right\|<\varepsilon$ for all $x\in F$.
 
 Since \Js{} is strongly self-absorbing, applying \cite[Theorem 2.2]{TW} we obtain a unitary $u$ in the multiplier algebra of $A\otimes\mathcal{Z}$ such that 
 \[\left\lVert\Ad u\circ\Phi(x) -x\otimes 1_{\mathcal{Z}}\right\rVert<\varepsilon/2\quad\text{for all }x\in F\cup\alpha(F).\]
  Then the desired endomorphism $\widetilde{\alpha}$ of $A$ can be obtained by $\widetilde{\alpha}=(\Ad u\circ\Phi)^{-1}\circ(\alpha\otimes\varphi)\circ\Ad u\circ\Phi$. 
 Actually, if there exists a conditional expectation for $\widetilde{\alpha}$, then $\alpha\otimes\varphi$ also has a conditional expectation $\widetilde{E}$. Taking a positive element $h\in A$ and a state $\omega$ of $A$ with $\omega(h)=1$, we have a conditional expectation $E$ for $\varphi$ by $E(x)=(\omega\otimes\id_{\mathcal{Z}})\circ\widetilde{E}(\alpha(h)\otimes x)$ for $x\in \mathcal{Z}$. Besides, it follows that $\|\widetilde{\alpha}(x)-\alpha(x)\|$
\vspace{-5pt}
\begin{align*} 
&\leq \|(\alpha\otimes\varphi)\circ\Ad u \circ\Phi(x)-(\alpha\otimes\varphi)(x\otimes 1_{\mathcal{Z}})\| 
 + \|\alpha(x)\otimes1_{\mathcal{Z}}-\Ad u\circ\Phi\circ\alpha(x)\| \\
 &<\varepsilon/2+\varepsilon /2 =\varepsilon\quad\text{for } x\in F.
 \end{align*}
 \end{proof}

\section{Non-transportable \Cs s}
Endomorphisms without conditional expectations were studied in \cite[Appendix B, Section 6]{EK} to consider the question: \emph{Is it true that the Cuntz algebra $\mathcal{O}_2$, or a unital separable infinite dimensional AF-algebra is transportable in $\mathcal{O}_2$?} Our main result can be applied to give a negative answer for $\mathcal{O}_2$ and simple AF-algebras, more generally for unital separable nuclear \Js-absorbing \Cs s.

We recall the definition of transportable \Cs s in $\mathcal{O}_2$ introduced by E. Kirchberg.

\begin{definition}\label{DefTrans}
For a unital separable \Cs{} $A$, a unital embedding $\iota$ of $A$ into $\mathcal{O}_2$ is called \emph{in general position} if there exists a unital embedding of $\mathcal{O}_2$ into the relative commutant \Cs{} $\Image(\iota)'\cap \mathcal{O}_2$ of the image of $\iota$.

A unital separable nuclear \Cs{} $A$ is called \emph{transportable in 
$\mathcal{O}_2$}, if for any two unital embeddings $\iota$ and $\eta$ of $A$ into $\mathcal{O}_2$ in general position there exists an automorphism $\alpha$ of $\mathcal{O}_2$ such that $\alpha\circ \iota=\eta$.
\end{definition}

\begin{theorem}\label{Thm4.2}
Any unital separable nuclear \Js-absorbing \Cs{} is not transportable in $\mathcal{O}_2$. In particular,\,all Kirchberg algebras, unital separable \Js-absorbing $AF$-algebras, and irrational rotation algebras are not transportable in $\mathcal{O}_2$. 
\end{theorem}
\begin{proof}
The famous Kirchberg's theorem \cite{K1} (see also \cite[Theorem 6.3.12]{Ror0},) allows us to obtain an embedding $\iota$ of $A$ into $\mathcal{O}_2$ and a conditional expectation $E_{\iota}$ for $\iota$. Replacing $\mathcal{O}_2$  with $\iota(1_A)\mathcal{O}_2\iota(1_A)\cong \mathcal{O}_2$, we may assume that $\iota$ is a unital embedding. Considering $\mathcal{O}_2\otimes\mathcal{O}_2\cong\mathcal{O}_2$ and a state of $\mathcal{O}_2$, we may further assume that $\iota$ is in general position. 
By Corollary \ref{Cor3.3} {\rm (i)}, there exists a unital endomorphism $\varphi$ of $A$ which does not have a conditional expectation. Define a unital embedding $\eta$ of $A$ into $\mathcal{O}_2$ by $\eta=\iota\circ\varphi$, which is also in general position. Note that $\eta$ does not have a conditional expectation. So, if there is an automorphism $\alpha$ of $\mathcal{O}_2$ satisfying $\alpha\circ\iota=\eta$, then $E_{\iota}\circ\alpha^{-1}$ becomes a conditional expectation for $\eta$, which is a contradiction. 
\end{proof}

\noindent{\bf Acknowledgements.}\quad  The author would like to thank  Professor Masaki Izumi, and Professor Narutaka Ozawa for helpful comments on this research. He also wish to express his gratitude to Professor Guihua Gong, Professor Huaxin Lin and the organizers of Special Week on Operator Algebras 2023.

\noindent Yasuhiko Sato \\
Graduate School of Mathematics, \\
Kyushu University, \\
744 Motoka, Nishi-ku, Fukuoka\\ 
819-0395\quad
Japan \\
e-mail: ysato@math.kyushu-u.ac.jp

\end{document}